\documentclass[12pt, a4paper]{article}

\usepackage[utf8]{inputenc}
\usepackage[T1]{fontenc}
\usepackage[english]{babel}
\usepackage{indentfirst}
\usepackage[top=2cm,bottom=2cm,right=2cm,left=2cm]{geometry}
\usepackage{amsmath}
\usepackage{amsfonts}
\usepackage{amssymb}
\usepackage{amsthm}
\usepackage{breqn}
\usepackage{graphicx}
\usepackage{babelbib}
\usepackage{lmodern} \normalfont

\newtheorem{thm}{Theorem}
\newtheorem{lem}{Lemma}

\newcommand{\naturais}{\mathbb{N}}
\newcommand{\Z}{\mathbb{Z}}
\newcommand{\R}{\mathbb{R}}
\newcommand{\A}{{\mathcal{A}}}
\newcommand{\B}{{\mathcal{B}}}
\newcommand{\I}{{\mathcal{I}}}
\newcommand{\PP}{\mathbb{P}}
\newcommand{\QQ}{{\mathbb{Q}}}
\newcommand{\V}{{\mathcal{V}}}
\newcommand{\W}{{\mathcal{W}}}
\newcommand{\EE}{{\mathbb{E}}}
\newcommand{\capa}{\mathop{\mathrm{cap}}}
\newcommand{\ome}{\boldsymbol{\omega}}
\newcommand{\ind}{{\bf 1}}

\begin{document}

\begin{center}
\LARGE Russo's formula for random interlacements
\end{center}

\begin{center}
\textit{Diego F. de Bernardini
} 
\hspace{1.7cm}  \textit{Serguei Popov \footnote{Department of Statistics, Institute of Mathematics, Statistics and Scientific Computation, University of Campinas, Brazil, 
e-mails: $\{$bernardini, popov$\}$@ime.unicamp.br}}
\end{center}

\vspace{0.1cm}

\begin{abstract}
\footnotesize 
In this paper we obtain a couple of explicit expressions for the derivative of the probability of an increasing event in the random interlacements model. The event is supported in a finite subset of the lattice, and the derivative is with respect to the intensity parameter of the model.

\vspace{0.3cm}


\end{abstract}

\begin{section}{Introduction}\label{intro}

The random interlacements process was introduced by Alain-Sol Sznitman~\cite{Szn10}, who was o\-ri\-gi\-nal\-ly motivated by a particular question about random walks and corrosion of materials. Basically, the motivation was the investigation of the trace of a simple random walk trajectory on large graphs. 

Since its introduction, this model has been extensively studied, see for example~\cite{Szn11P,Szn11A,Tei09,Tei11}. Besides Sznitman's original article~\cite{Szn10}, other good references on this subject are, for example, \cite{CT12,DRS14}.

Roughly, the random interlacements model is characterized by a Poisson point process on a space of doubly infinite simple random walk trajectories in~$\Z^d$, with~$d\geq 3$, so that a realization of the process is basically a random ``Poissonian'' soup of such trajectories. The model also has a positive intensity parameter~$u$ which controls this soup, in such a way that more trajectories are included in the process as the value of the parameter increases. 

The main difficulty when studying this process stems from the fact that e.g.\ the set of sites covered by the trajectories is a dependent field.
The dependence and decoupling properties of the process were extensively investigated in recent works, see for example~\cite{PT,Szn10}.

Specifically, the goal of this article is to establish expressions for the derivative, with respect to the intensity parameter, of the probability of an increasing event in the random interlacements model, when this event is supported in a finite subset of~$\Z^d$. We will call these expressions Russo's formula for random interlacements, in analogy to the corresponding formula that comes from percolation theory and which establishes an expression for the derivative of the probability of increasing events in the usual percolation model (see e.g.\ Theorem~2.25 in Section~2.4 of~\cite{Gri99}).

This paper is organized in the following way: 
In Section~\ref{definitions}, we recall the definition of the random interlacements model, as well as some of its characteristics, and we also discuss other related definitions, as the increasing events and the plus-pivotal trajectories. Then, we present our result in Section~\ref{results}, and discuss some possible applications in Section~\ref{s_applications}. Finally, the proofs are placed in Section~\ref{proofs}.

\end{section}

\begin{section}{Definitions: The random interlacements process}\label{definitions}

In this section, we recall the definition of the random interlacements process, introduced by Alain-Sol Sznitman in~\cite{Szn10}.

We need to consider first the spaces of infinite and doubly infinite trajectories in~$\Z^d$, $d\geq 3$, respectively defined as
\begin{align*}
W_+&= \Big\{ w:\naturais\rightarrow\Z^d : \|w(n+1)-w(n)\|=1, \forall n\in\naturais, \\  
& \hspace{6cm}\mbox{ and } \#\{n:w(n)=y\}<\infty, \forall y\in\Z^d\Big\}, \\
\mbox{and } \ \  W&= \Big\{ w:\Z\rightarrow\Z^d : \|w(n+1)-w(n)\|=1, \forall n\in\Z, \\  
& \hspace{6cm}\mbox{ and } \#\{n:w(n)=y\}<\infty, \forall y\in\Z^d\Big\},
\end{align*}
where $\|\cdot\|$ denotes the Euclidean norm. These spaces are respectively endowed with the~$\sigma$-al\-ge\-bras $\W_+$ and~$\W$, generated by their canonical coordinates~$\{X_n\}_{n\in\naturais}$ and~$\{X_n\}_{n\in\Z}$.

Additionally, define the space of trajectories modulo time-shift,
\begin{align*}
 W^* = W/\sim, \ \ \mbox{where} \ \ w\sim w' \Leftrightarrow w(\cdot)= w'(\cdot + k), \ \ \mbox{for some} \ \  k\in\Z, 
\end{align*}
endowed with the~$\sigma$-algebra~$\W^* = \{B\subset W^*:(\pi^*)^{-1}(B)\in\W\}$, where~$\pi^*$ is the canonical projection from the space~$W$ to~$W^*$.

For a finite set $G\subset \Z^d$, 
define its internal boundary~$\partial G := \{x\in G : \|x-y\|=1 \mbox{ for some } y\notin G\}$, and the stopping time~$\tilde{H}_G(w):=\min\{n\geq 1 : X_n(w)\in G\}$ for~$w\in W_+$, assuming that~$\min\{\emptyset\}=\infty$. Also, denote by~$W_G$ the set of trajectories in~$W$ which necessarily visit the set~$G$, $W_G := \{w\in W : X_n(w)\in G, \mbox{ for some }  n\in\Z\}$, so that, by the definition of~$\pi^*$, the set of modulo time-shift trajectories that visit~$G$ is given by~$ W_G^* := \pi^*(W_G)$.

It is then possible to define the harmonic measure in~$G$, $e_G(x):= P_x[\tilde{H}_G=\infty]\ind_G(x)$, for~$x\in\Z^d$, where~$\ind_G$ denotes the indicator function in~$G$ and~$P_x[~\cdot~]$ is the law of the simple random walk starting at~$x$. Thus, the capacity of the set~$G$ is defined by~$\capa(G):=\sum_{x\in\Z^d}e_{G}(x)$, and the normalized harmonic measure by $\bar{e}_G(x) := e_G(x)/\capa(G)$, for~$x\in \Z^d$.

Finally, the random interlacements process is governed by a Poisson point process on the measurable space~$(W^*\times\R_+, \mathcal{W}^*\otimes\B\mbox{(}\R_+\mbox{)})$, with a specific intensity measure. To describe this intensity measure, consider the measure denoted by~$Q_G$, which is defined on~$(W,\W)$, such that
\begin{align*}
Q_G\Big((X_{-n})_{n\geq 0}\in B_1, X_0=x, (X_{n})_{n\geq 0}\in B_2\Big) = P_x(B_1\mid \tilde{H}_G=\infty)e_G(x)P_x(B_2),
\end{align*}
for any~$B_1,B_2\in\W_+$ and~$x\in\Z^d$. The above mentioned intensity measure is just the product measure~$\nu\otimes\lambda_+$, where~$\lambda_+$ denotes the Lebesgue measure on~$\R_+$ and~$\nu$ is the only~$\sigma$-finite measure in~$(W^*,\W^*)$ such that $\ind_{W_G^*}\cdot\nu = \pi^*\circ Q_G$, for any finite set~$G\subset\Z^d$, where~$\ind_{W_G^*}\cdot\nu(\cdot) := \nu(W_G^*\cap\cdot)$, see Theorem~1.1 of~\cite{Szn10} where the existence and the uniqueness of~$\nu$ are established.

Let us consider also the space of locally finite point measures on~$W^*\times\R_+$,
\begin{align*}
\Omega &= \Big\{ \ome=\sum_{i\geq1}\delta_{(w_i^*,u_i)} : w_i^*\in W^*, u_i\in\R_+, \mbox{ such that }   \\  
&\hspace{3.5cm}\ome(W_G^*\times[0,u])<\infty, \mbox{ for every finite } G\subset\Z^d \mbox{ and } u\geq 0 \Big\},
\end{align*}
endowed with the~$\sigma$-algebra~$\A$ generated by the mappings~$\ome\mapsto\ome(D)$, for~$D\in\W^*\otimes\B(\R_+)$, and let us denote by~$\PP$ the law of the Poisson point process on ($\Omega$, $\A$) with intensity measure~$\nu\otimes\lambda_+$, that characterizes the random interlacements process.

Now, for~$u\geq 0$, we denote by~$\PP^u$ the law of the Poisson point process which governs the random interlacements process at level~$u$, restricted to the set~$G$. Observe that we are omitting the dependence on the set~$G$ in this notation. 

In words, in the interlacements process restricted to~$G$ at level~$u$, a Poisson-distributed random variable with parameter~$u\capa(G)$ determines the number of independent simple random walks which are started at the boundary~$\partial G$, where each one of the starting sites is randomly chosen according to the measure~$\bar{e}_G(x)$, for~$x\in\partial G$. Then, the walks are let run up to infinity.

Precisely, $\PP^u$ is the law of a Poisson point process on~$W_+$ with intensity measure equal to~$uP_{e_G}$, where, for~$B\in\W_+$, $P_{e_G}(B):=\sum_{x\in \Z^d}e_G(x)P_x(B)$. In the general case (that is, at any level) the process restricted to~$G$ is described by a Poisson point process on~$W_+\times\R_+$ with intensity measure~$P_{e_G}\otimes\lambda_+$.

To conclude the discussion, analogously to the definition of~$\Omega$ consider now the space of locally finite point measures on~$W_+\times\R_+$,
\begin{align*}
\Omega_+ = \Big\{ \ome_+=\sum_{i\geq1}\delta_{(w_i,u_i)} : w_i\in W_+, u_i\in\R_+, \mbox{ such that }  \ome_+(W_+\times[0,u])<\infty, \mbox{ for all } u\geq 0 \Big\},
\end{align*}
endowed with the~$\sigma$-algebra~$\A_+$ generated by the mappings~$\ome_+\mapsto\ome_+(D)$, for~$D\in\W_+\otimes\B(\R_+)$.

Then, consider the space of one-sided trajectories in~$\Z^d$ which necessarily start at~$\partial G$, $W_{G,+} := \{w\in W_+ : X_0(w)=w(0)\in\partial G\}$, and with that, define also the spaces~$\Omega^{G,+}:= \{\ome\in~\Omega_+ : \ome=\sum_{i\geq 1}\delta_{(w_i,u_i)}, w_i\in W_{G,+}\}$ and~$\Omega^{G,+}_u:=\{\ome\in\Omega^{G,+} : \ome=\sum_{i\geq 1}\delta_{(w_i,u_i)}, u_i\leq u \}$. We denote by~$\ome_u$ an element of~$\Omega^{G,+}_u$, and we interpret~$\ome_u$ as the random realization (under~$\PP^u$) of the interlacements process at level~$u$, restricted to~$G$, that is, the random configuration of infinite trajectories with indexes smaller than or equal to~$u$, which start at~$\partial G$.

It is worth to mention a partial order relation between the elements of~$\Omega^{G,+}$. Precisely, for configurations~$\ome$ and~$\ome'$ in~$\Omega^{G,+}$, we write~$\ome\leq\ome'$ whenever all trajectories composing~$\ome$ are also present in~$\ome'$. Thus, if the trajectory~$w_i\in W_{G,+}$ (with index~$u_i$) is present in the realization of the process (restricted to~$G$) at level~$u$, then it will also be present in the realization of the process at all levels~$u'\geq u$, and we write~$\ome_u\leq\ome_{u'}$ whenever~$u\leq u'$. 

Lastly, for~$\ome_u=\sum_{i\geq 1}\delta_{(w_i,u_i)}\in\Omega_u^{G,+}$, we recall the definitions of the interlacement and the vacant sets, restricted to~$G$ at level~$u$, respectively given by 
\begin{align} 
\I^u_G = \I^u_G(\ome_u) = \Big(\bigcup_{i\geq 1} R(w_i)\Big)\cap G \ \ \mbox{ and } \ \ \V^u_G= \V^u_G(\ome_u) = G\setminus \I^u_G(\ome_u), 
\label{eq21}
\end{align}
where~$R(w_i)$ is the range of~$w_i$.

\begin{subsection}{Increasing events}\label{inc_ev}

We now discuss the notion of increasing events in the random interlacements model. 
Formally, an event~$A$ is said to be increasing with respect to the random interlacements process restricted to~$G$ if, for~$\ome,\ome'\in\Omega^{G,+}$, one has~$\ind_A(\ome)\leq\ind_A(\ome')$ whenever~$\ome\leq\ome'$, that is, 
if the event~$A$ occurs under configuration~$\ome$, then~$A$ also has to occur under configuration~$\ome'$, whenever~$\ome\leq\ome'$. On the other hand, as in the case of Bernoulli (site) percolation, one can also talk about an event to be increasing with respect to the sites of~$\Z^d$ (or the sites of~$G$) in the usual way, that is, if the event occurs under a certain configuration of ``visited'' (or ``open'') sites, then it will also occur if more sites are visited. The first notion (which refers to the trajectories of the interlacement) is more general, in the sense that, if an increasing event is described in terms of vacant/visited sites of the lattice, then the corresponding event defined through the trajectories will also be increasing. Thus, we will always refer to the first notion, that is, we consider increasing events with respect to the (trajectories of the) random interlacements.

Moreover, we say that the event is supported on the set~$G$ when it is defined only in terms of the sites in~$\Z^d$ which belong to~$G$. More precisely, consider the~$\sigma$-algebra~$\W_{G,+}$ in~$W_{G,+}$ generated by its canonical coordinates, and the~$\sigma$-algebra~$\A^{G,+}$ of subsets of~$\Omega^{G,+}$, generated by the mappings~$\ome\mapsto\ome(D)$, for $D\in\W_{G,+}\otimes\B(\R_+)$. Thus, to say that the event is supported on the set~$G$ means that this event belongs to the $\sigma$-algebra~$\A^{G,+}$ of subsets of~$\Omega^{G,+}$, and so represents a collection of configurations~$\ome$ of trajectories that start on the boundary of~$G$, where these configurations~$\ome$ on the mentioned collection satisfy some condition imposed only in terms of the sites of the set~$G$.

An example can be constructed in the following way. Denote by~$\Psi$ the class of finite paths of neighbor sites in~$\Z^d$, precisely defined by
\begin{align*} 
\Psi = \Big\{ \tau:\{0,1,\dots,k\}\rightarrow\Z^d : k\geq 1 \ \ \mbox{and} \ \ \|\tau(n+1)-\tau(n)\|=1, \ \ \forall n=0,1,\dots,k-1\Big\}. 
\end{align*}
Fix two distinct sites~$v$ and~$z$ in~$G$.
Then, the following event will be increasing
\begin{center}
\textit{``there exist a finite path (in~$\Psi$), completely contained in the interlacement set restricted to~$G$ at level~$u$, connecting~$v$ and~$z$'',   }
\end{center}
which can formally be represented as
\begin{align*} 
\Big\{\ome\in \Omega^{G,+}: \exists \tau\in\Psi \ \ \mbox{with}\ \  R(\tau)\subset\I^u_G, \ \ \tau(0)=v \ \ \mbox{ and }\ \  \tau(k)=z\Big\}, 
\end{align*}
for~$v,z\in G$, where~$R(\tau)$ represents, as before, the range of the path~$\tau$, see Figure~\ref{figure1} for an illustration.
\begin{figure}[ht]
\centering
\includegraphics[scale=0.5]{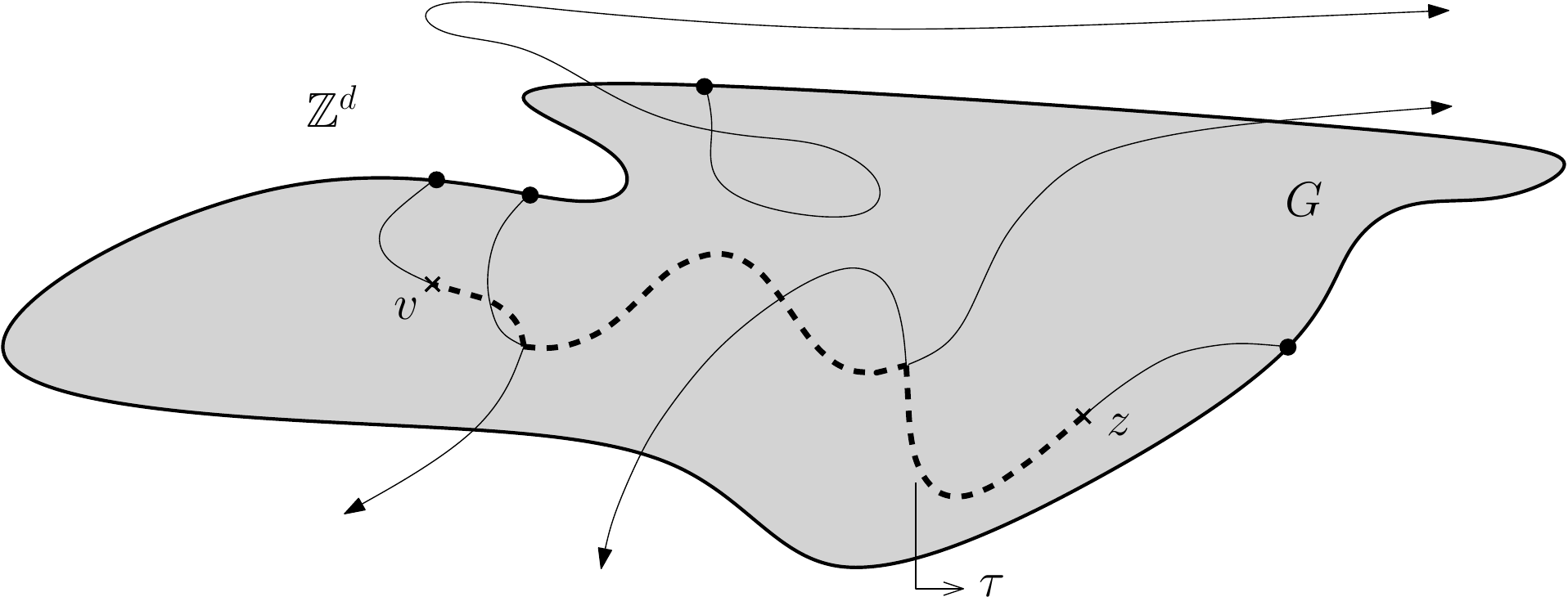}
\caption{\footnotesize In this configuration, the sites~$v$ and~$z$ in~$G$ are connected by a finite path~$\tau$ (dashed line), completely contained in~$\I_G^u$.}
\label{figure1}
\end{figure}

\end{subsection}

\begin{subsection}{Plus-pivotal trajectories}\label{ppt}

For an increasing event~$A$, we say that the trajectory~$w^*\in W_{G,+}$ (with respective index~$v\leq u$) is plus-pivotal for the event~$A$ on the configuration~$\ome_u\in\Omega^{G,+}_u$ if the event~$A$ occurs and~$w^*$ is present in the configuration~$\ome_u$, but~$A$ would no longer occur if~$w^*$ was removed from~$\ome_u$, that is,~$\ind_A(\ome_u)=1$ and~$\ind_A(\ome_u-\delta_{(w^*,v)})=0$. For example, three of the four trajectories on Figure~\ref{figure1} are plus-pivotal for the event which was mentioned in that example. On that figure, except for the upper trajectory, if any of the other three trajectories is removed, we will not have a path contained in~$\I_G^u$ connecting the sites~$v$ and~$z$ anymore.

We denote by~$N^{+}_{\ome_u}$ the number of plus-pivotal trajectories for the increasing event~$A$, on configuration~$\ome_u$. Observe that, since the plus-pivotal trajectories only exist when~$A$ occurs (under~$\ome_u$), we can then write $ N^{+}_{\ome_u} = N^{+}_{\ome_u}\ind_A(\ome_u)$. Note that we are omitting the dependence on the event~$A$ in this notation. 

Recall that, in the Bernoulli (bond) percolation model in~$\Z^d$ ($d\geq 1$), a bond (or edge) of~$\Z^d$ is said to be pivotal for an increasing event if and only if the event occurs when that bond is ``open'' and it does not occur when the same bond is ``closed'', keeping unchanged the states of all other bonds.  

\end{subsection}
\end{section}

\begin{section}{Result}\label{results}

We present now the result we have obtained, concerning the probability of an increasing event supported on a finite set of~$\Z^d$, in the random interlacements model restricted to that set, and in the next section we develop its proof.

In order to study the derivative of the above mentioned probability, it is possible to show first that this probability is analytic (and hence differentiable) as a function of the intensity parameter~$u$. More precisely, it is possible to show that, if~$G$ is a finite subset of~$\Z^d$, with~$d\geq 3$, and~$A$ is an event supported on~$G$ and increasing with respect to the interlacement~$\I^u_G$, then the probability of~$A$ under the law~$\PP^u$, denoted by~$\PP^u(A)$, is an analytic function of~$u$.

Since the formal proof for this statement is elementary, we give only a brief idea on how to do it. Basically, by using the total probability formula, conditioning on the number of trajectories in the interlacement~$\I^u_G$ (which is a Poisson random variable), one can write the probability~$\PP^u(A)$ as the product of an exponential function, which is analytic, by a power series on~$u$, which we can also show that is analytic as a function of~$u$.

Next, before stating our main result, we need to introduce some notations:

\begin{itemize}
\item $\EE^u$ represents the expectation under the law~$\PP^u$;
\item $\QQ$ denotes the law of a trajectory~$w_{\eta}$ in~$W_{G,+}$, with the starting point chosen according to the normalized harmonic measure~$\bar{e}_G$;
\item $\widehat{\EE}^u$ represents the expectation under the product measure~$\PP^u\otimes \QQ$;
\item and finally, $M$ denotes the number of trajectories in~$\I^u_G$, observing that it has a Poisson distribution with parameter~$u\capa(G)$.
\end{itemize}

Also, recall that~$N^+_{\ome_u}$ is the number of plus-pivotal trajectories for the event~$A$ on the configuration~$\ome_u$ in~$\Omega_u^{G,+}$, as defined in Section~\ref{ppt}.

\begin{thm}[Russo's formula for random interlacements]
 \label{RFRI_Thm}
Let~$G$ be a finite subset of~$\Z^d$, with~$d\geq3$, and also let~$A$ be an increasing event in~$G$. Then, for $u>0$
\begin{align}
 \frac{d}{du}\PP^u(A) &= \capa(G)\widehat{\EE}^u\Big[\ind_{\{\ome_{u} +\delta_{(w_\eta)}\in A\}}\ind_{\{\ome_u\notin A\}}\Big] \label{RFRI_1}  \\
&= \frac{1}{u}\EE^u\Big[N^{+}_{\ome_u}\Big] \label{RFRI_2} \\
&= \frac{1}{u}\PP^u(A)\Big( \EE^u\Big[ M\mid \{\ome_u\in A\} \Big] - u\capa(G) \Big). \label{RFRI_3}
\end{align}
Moreover, \eqref{RFRI_1} also gives the expression for the right derivative at $u=0$.
\end{thm}

In words, in the right-hand side of~\eqref{RFRI_1} we have the probability that the event does not occur at level~$u$, but adding one more trajectory causes it to occur. Particularly at~$u=0$, it will be equal to zero if~$A$ is trivial, and it will be just the probability that~$A$ occurs under the presence of only one trajectory if~$A$ is not trivial. Also, it is worth mentioning that the term~$u\capa(G)$ in~\eqref{RFRI_3} is the unconditional expectation of~$M$. 

\end{section}

\begin{section}{Some applications} \label{s_applications}

In order to exhibit a first application of the expressions appearing in Theorem~\ref{RFRI_Thm}, observe that, if~$X_1$ and~$Y$ are two random variables with a Poisson distribution with parameter~$u\capa(G)$, and if a third random variable~$X_2$ is defined to be equal to~$Y+1$, then the total variation distance between the law~$\PP^u$ of the random interlacements process restricted to~$G$, at level~$u$, and the law~$\PP^u\otimes\QQ$ of the same process with one additional independent trajectory, say~$w_\eta$, satisfies
\begin{align}
\| \PP^u - \PP^u\otimes\QQ \|_{\mathop{\mathrm{TV}}} \leq \| X_1 - X_2 \|_{\mathop{\mathrm{TV}}},
\label{eq51}
\end{align}
where $\|\cdot \|_{\mathop{\mathrm{TV}}}$ is the total variation norm. Indeed, if the coupling of $X_{1,2}$ is successful, this yeilds a coupling of~$\PP^u$ and~$\PP^u\otimes\QQ$ by simply making the particles follow the same trajectories. 

We need the following elementary fact:
\begin{lem}
\label{lemma51}
For the random variables~$X_1$ and~$X_2$ defined as above, it holds that
\begin{align*}
\| X_1 - X_2 \|_{\mathop{\mathrm{TV}}} \leq \Big(u\capa(G)\Big)^{-1/2}.
\end{align*}
\end{lem}

\begin{proof}
To simplify the notation, abbreviate~$u\capa(G)=\theta$, which can be seen as a generic parameter for the Poisson distribution. Then, we have
\begin{align*}
\| X_1 - X_2 \|_{\mathop{\mathrm{TV}}} =
e^{-\theta}+\sum_{k=1}^{\infty} \Big|\frac{e^{-\theta}\theta^k}{k!}
-\frac{e^{-\theta}\theta^{k-1}}{(k-1)!}\Big| 
= \sum_{k=0}^{\infty} \Big|\frac{k}{\theta}-1\Big| \times
\frac{e^{-\theta}\theta^k}{k!} = E\Big|\frac{X_1}{\theta}-1\Big|.
\end{align*}
But 
\begin{align*}
E\Big|\frac{X_1}{\theta}-1\Big| 
= \frac{1}{\theta}E|X_1-\theta| \leq \frac{1}{\theta}
\sqrt{E|X_1-\theta|^2} = \theta^{-1/2},
\end{align*}
which shows the claim.
\end{proof}

Now observe that, since the event~$A$ in Theorem~\ref{RFRI_Thm} is increasing, then~$\{\ome_u\in A\}\subset\{\ome_{u} +\delta_{(w_\eta)}\in A\}$. Therefore the expectation in expression~\eqref{RFRI_1} will be
\begin{align*}
\widehat{\EE}^u\Big[\ind_{\{\ome_{u} +\delta_{(w_\eta)}\in A\}}\ind_{\{\ome_u\notin A\}}\Big] &= \PP^u\otimes\QQ\Big[\ome_{u} +\delta_{(w_\eta)}\in A, \ome_u\notin A\Big] \\
&= \PP^u\otimes\QQ\big[A\big] - \PP^u\big[A\big] \\
&\leq \| \PP^u - \PP^u\otimes\QQ \|_{\mathop{\mathrm{TV}}}.
\end{align*}

From the above, along with~\eqref{eq51}, Lemma~\ref{lemma51} and~\eqref{RFRI_1}, for increasing events we obtain the following \emph{universal} upper bound on the derivative:
\begin{align*}
\frac{d}{du}\PP^u(A) \leq \sqrt{\frac{\capa(G)}{u}}.
\end{align*}
Also, from~\eqref{RFRI_2} we have the following upper bound on the expected number of plus-pivotal trajectories for the event~$A$ on configuration~$\ome_u$, 
\begin{align}
\EE^u\Big[N^{+}_{\ome_u}\Big] \leq \sqrt{u\capa(G)}.
\label{est_pivotal}
\end{align}
Upper bounds on the number of plus-pivotal trajectories may prove useful in diverse situations. For example, assume that we know what is going on in some region, and want to know, how can it affect the occurrence of some (increasing)
event in another (distant) region. Assume also that it is known that only with a small probability there is a trajectory that crosses both regions. Then, one may note that, even if such a trajectory exists, it is probably not pivotal for the event in the second region, since the (relative) number of pivotal trajectories cannot be large, as~\eqref{est_pivotal} shows.

In fact, \eqref{est_pivotal} can be further improved if one is allowed to vary the parameter~$u$. 
Indeed, \eqref{RFRI_2} implies that, for~$0\leq u_1<u_2<\infty$, 
\begin{align*}
\int_{u_1}^{u_2} \frac{1}{u}\EE^u\Big[N^{+}_{\ome_u}\Big] du
=\int_{u_1}^{u_2} \frac{d}{du}\PP^u(A) du = \PP^{u_2}(A)-\PP^{u_1}(A)\leq 1.
\end{align*} 
Then, for some constant~$\alpha>1$, define the following subset of~$[u_1,u_2]$,
\begin{align*}
D_{u_1,u_2,\alpha} = \Big\{ u\in [u_1,u_2] ~:~ \frac{1}{u}\EE^u\Big[N^{+}_{\ome_u}\Big] \leq \frac{\alpha}{u_2-u_1}\Big\},
\end{align*}
and observe that the Lebesgue measure of its complementary in the interval~$[u_1,u_2]$ is
\begin{align*}
\lambda_+\Big([u_1,u_2]\setminus 
D_{u_1,u_2,\alpha}\Big) = \int_{u_1}^{u_2} 
\ind_{[u_1,u_2]\setminus D_{u_1,u_2,\alpha}}(u) du 
< \frac{u_2-u_1}{\alpha} \int_{u_1}^{u_2}\frac{1}{u}\EE^u\Big[N^{+}_{\ome_u}\Big] du \leq \frac{u_2-u_1}{\alpha},
\end{align*}
so that
\begin{align}
\lambda_+\Big(D_{u_1,u_2,\alpha}\Big) > (u_2-u_1)\Big(1-\frac{1}{\alpha}\Big).
\label{eq52}
\end{align}
Thus we have
\begin{align*}
\EE^u\Big[N^{+}_{\ome_u}\Big] \leq \frac{u\alpha}{u_2-u_1},
 \mbox{ for all } u \mbox{ in } D_{u_1,u_2,\alpha},
\end{align*}
where~$D_{u_1,u_2,\alpha}$ satisfies~\eqref{eq52}. In words, regardless on the size of the set where the event takes place, the expected number of plus-pivotal trajectories in ``most'' (in the sense of~\eqref{eq52}) points 
of the interval~$[u_1,u_2]$ cannot exceed a quantity depending on~$u_2-u_1$ and~$\alpha$.

To mention another possible application, consider the expression~\eqref{RFRI_3}, and note that, when it is possible to obtain an explicit expression for the probability~$\PP^u(A)$ (and for its derivative), we can establish an explicit expression for the expected number of trajectories in the interlacements restricted to~$G$, at level~$u$, conditioned on the occurrence of the event~$A$ at level~$u$, 
\begin{align*}
\EE^u\Big[ M\mid \{\ome_u\in A\} \Big] = u\capa(G) + u\frac{\frac{d}{du}\PP^u(A)}{\PP^u(A)}, 
\end{align*}
recalling that~$u\capa(G)$ is the (unconditional) expectation of~$M$.

For example, consider the increasing event~$A=\{G\neq\V^u_G\}=\{\I^u_G\neq\emptyset\}$, where the interlacement and the vacant sets restricted to~$G$ at level~$u$, $\I^u_G$ and~$\V^u_G$, are given in~\eqref{eq21}. In this case, $\PP^u(A) = 1- \PP^u(G=\V^u_G) = 1- e^{-u\capa(G)}$,
so
\begin{align*}
\EE^u\Big[ M\mid \{\ome_u\in A\} \Big] = \frac{u\capa(G)}{1-e^{-u\capa(G)}}. 
\end{align*}

\end{section}

\begin{section}{Proof of Theorem~\ref{RFRI_Thm}}\label{proofs}

From the analyticity of~$\PP^u(A)$, we can naturally conclude that this probability, as a function of~$u$, is indeed differentiable on~$[0,\infty)$. So, in order to prove Theorem~\ref{RFRI_Thm}, it will be sufficient to compute the right and left derivatives
\begin{align*}  
\displaystyle\lim_{h\downarrow 0}\frac{\PP^{u+h}(A)-\PP^u(A)}{h} \ \ \ \mbox{  and  } \ \ \ \displaystyle\lim_{h\downarrow 0}\frac{\PP^u(A)-\PP^{u-h}(A)}{h}, 
\end{align*}
which are automatically equal.

Although it is possible to prove the theorem in another way, for example by first obtaining just one of that expressions for the derivative and then showing that the others are equal to this one, we chose to compute the side derivatives because we believe it is more instructive.

\begin{subsection}{The right derivative}\label{right_derivative}

We begin with the right derivative. 
Recall that~$A$ is an increasing event with respect to~$\I^u_G$, supported on the finite set~$G\subset\Z^d$.
For some~$u> 0$ and $h>0$, we first compute the probability~$\PP^{u+h}(A)$. We have 
\begin{align*} 
\PP^{u+h}(A) = \EE^u\Big[\PP^{u+h}(A\mid \ome_u)\Big] = \EE^u\Big[\PP(\ome_{u+h}\in A\mid \ome_u)\Big].
\end{align*}

Since the event~$A$ is increasing, if~$\ome_u\in A$ then by definition we also have~$\ome_{u+h}\in A$, so
\begin{align*}
\PP^{u+h}(A\mid \ome_u)=\PP(\ome_{u+h}\in A \mid  \ome_u)=1, \mbox{ on } \{\ome_u\in A\}.
\end{align*}

On the other hand, in order to establish an expression for~$\PP(\ome_{u+h}\in A \mid  \ome_u)$ on~$\{\ome_u\notin A\}$, we observe that it is possible to describe the coupling of the interlacements processes restricted to~$G$, for all levels~$u\geq 0$, by considering a Poisson process in~$\mathbb{R_+}$ with intensity equal to~$\capa(G)$, in such a way that the number of random walks trajectories in the process at level~$u$ will be given by the number of points of this Poisson process on the interval~$[0,u]$ (see Figure~\ref{figure2}). Recall that we denote this number by~$M$, and~$M\sim\mbox{Poisson}(u\capa(G))$. 
\begin{figure}[ht]
\centering
\includegraphics[scale=0.7]{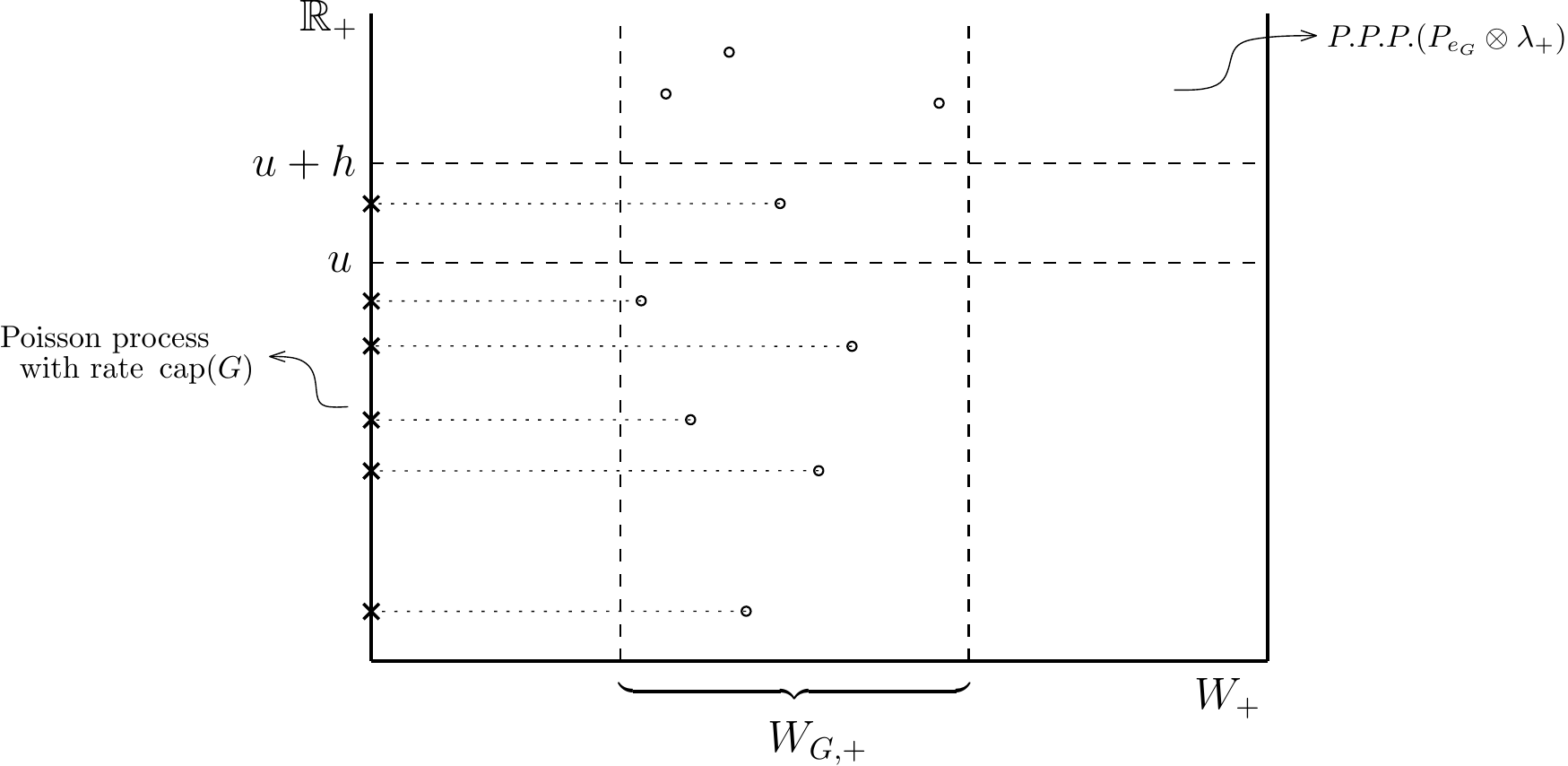}
\caption{\footnotesize The coupling of the interlacements processes restricted to~$G$, for all non-negative levels. In this picture, by increasing the level from~$u$ to~$u+h$, only one trajectory is added to the configuration. Recall the spaces~$W_+$ and~$W_{G,+}$, defined in Section~\ref{definitions}. Here, we write~$P.P.P.(P_{e_G}\otimes\lambda_+)$ for the Poisson point process on~$W_+\times\R_+$ with intensity measure~$P_{e_G}\otimes\lambda_+$.}
\label{figure2}
\end{figure}

So, by increasing the level of the process (restricted to~$G$) from~$u$ to~$u+h$, it is possible that some new trajectories are included in configuration~$\ome_u$, thus forming configuration~$\ome_{u+h}$ (see Figure~\ref{figure2} again). Then, denote by~$H$ the number of such trajectories that are eventually included in this case, noting that the value of~$H$ is given by the number of points of that Poisson process on the interval~$(u, u+h]$, so that~$H\sim\mbox{Poisson}(h\capa(G))$.

Observe now that, given the event~$\{H=0\}$, the configuration~$\ome_{u+h}$ will be equal to~$\ome_{u}$, whereas given the event~$\{H=1\}$, configuration~$\ome_{u+h}$ will be of the form~$\ome_u+\delta_{(w_\eta,u_\eta)}$, where~$w_\eta$ belongs to~$W_{G,+}$ and represents the (only) trajectory which is added to configuration~$\ome_u\in\Omega_u^{G,+}$, with an index~$u_\eta$. Since index~$u_\eta$ is irrelevant in this context, in the sense that we are interested only on the configuration of trajectories that is obtained after the inclusion of~$w_\eta$, then we write only~$\ome_u+\delta_{(w_\eta)}$ rather than~$\ome_u+\delta_{(w_\eta,u_\eta)}$ from now on.

Then, we can write
\begin{align*}
\PP(\ome_{u+h}\in A\mid \ome_u) = \sum_{k\geq0} \PP(\ome_{u+h}\in A\mid H=k,\ \ \ome_u)\PP(H=k\mid \ome_u). 
\end{align*}

But, taking into account the above observations about conditioning on the value of~$H$, we can conclude that, on~$\{\ome_u\notin A\}$, $\PP(\ome_{u+h}\in A\mid H=0, \ome_u) = 0$, and also
\begin{align*}
\PP(\ome_{u+h}\in A\mid H=1,\ \ \ome_u) = \PP(\ome_{u} +\delta_{(w_\eta)}\in A\mid H=1,\ \ \ome_u) = \QQ(\ome_{u} +\delta_{(w_\eta)}\in A),
\end{align*}
where the last expression represents the probability of the event~$\{\ome_{u} +\delta_{(w_\eta)}\in A\}$, under the law of the trajectory~$w_\eta$.

In addition, observe that, for any~$k\geq 0$, the probabilities~$\PP(\ome_{u+h}\in A\mid H=k,\ \ \ome_u)$ do not depend on~$h$, but only on~$\ome_u$ and on the number~$k$ of points of the above mentioned Poisson process in~$(u, u+h]$. So, we write $\PP(\ome_{u+h}\in A\mid H=k,\ \ \ome_u) := a_k^{\ome_u}$ for all $k\geq 0$, in order to simplify the notation.

Thus, since~$\PP(H=k\mid \ome_u)=\PP(H=k)$ for all~$k\geq 0$, and~$\PP(H=1)=h\capa(G)e^{-h\capa(G)}$, we have, on the event~$\{\ome_u\notin A\}$,
\begin{align*}
\PP^{u+h}(A\mid \ome_u) = \QQ(\ome_{u} +\delta_{(w_\eta)}\in A)h\capa(G)e^{-h\capa(G)} + R_h, 
\end{align*}
where 
\begin{align*} 
R_h = \sum_{k\geq2} a_k^{\ome_u}\PP(H=k) = \sum_{k\geq2} a_k^{\ome_u}\frac{e^{-h\capa(G)}(h\capa(G))^k}{k!}. 
\end{align*}

It is then possible to write
\begin{align*}
\PP^{u+h}(A\mid \ome_u) =  \QQ(\ome_{u} +\delta_{(w_\eta)}\in A)h\capa(G)e^{-h\capa(G)}\ind_{\{\ome_u\notin A\}} + R_h\ind_{\{\ome_u\notin A\}} + \ind_{\{\ome_u\in A\}}, 
\end{align*}
and taking the expectation under the law~$\PP^u$, we finally have 
\begin{align*}
\PP^{u+h}(A) = \EE^u\Big[ \QQ(\ome_{u} +\delta_{(w_\eta)}\in A)\ind_{\{\ome_u\notin A\}}\Big]h\capa(G)e^{-h\capa(G)} + \EE^u\Big[R_h\ind_{\{\ome_u\notin A\}}\Big] + \PP^u(A).
\end{align*}

But using the Monotone Convergence Theorem, it is elementary to show that the expectation of~$R_h\ind_{\{\ome_u\notin A\}}$ is equal to~$o(h)$ when~$h\downarrow 0$, and so we can conclude that
\begin{align*}
\displaystyle\lim_{h\downarrow 0}\frac{\PP^{u+h}(A)-\PP^u(A)}{h} = \capa(G)\EE^u\Big[ \QQ(\ome_{u} +\delta_{(w_\eta)}\in A)\ind_{\{\ome_u\notin A\}}\Big].
\end{align*}

In order to rewrite this last expression in a more intuitive way, simply note that
\begin{align*}
\EE^u\Big[ \QQ(\ome_{u} +\delta_{(w_\eta)}\in A)\ind_{\{\ome_u\notin A\}}\Big] = \widehat{\EE}^u\Big[\ind_{\{\ome_{u} +\delta_{(w_\eta)}\in A\}}\ind_{\{\ome_u\notin A\}}\Big],
\end{align*}
recalling that~$\widehat{\EE}^u$ is the expectation under the product measure~$\PP^u\otimes \QQ$. This proves expression~\eqref{RFRI_1} for positive values of~$u$.

Now, for~$u=0$ and~$h>0$, it is clear that, if~$A$ is trivial then~$\PP^h(A)=\PP^0(A)=1$, so that the right derivative will be null, which coincides with~\eqref{RFRI_1} in this case. On the other hand, if~$A$ is not trivial, it is elementary to use the same arguments as above to see that the right derivative will again be given by the same expression (with~$u=0$), that is, we have the capacity of~$G$ times the probability that~$A$ occurs under the presence of just one trajectory of the random interlacements. Thus, expression~\eqref{RFRI_1} is valid also for~$u=0$.

\end{subsection}

\begin{subsection}{The left derivative}\label{left_derivative}

Let~$u> 0$ and~$0<h<u$. 
We have
\begin{align*} 
\PP^{u-h}(A) = \PP(\ome_{u-h}\in A) = \EE^u\Big[\PP(\ome_{u-h}\in A\mid \ome_u)\Big],
\end{align*}
and since the event~$A$ is increasing, then~$\PP(\ome_{u-h}\in A \mid \ome_u) = 0$ on~$\{\ome_u\notin A\}$.

Consider again the random variable~$M$ which represents the number of trajectories in configuration~$\ome_u$, recalling that we interpret it as the number of points of the Poisson process with intensity~$\capa(G)$ in the interval~$[0,u]$. Thus we know that, given the value of~$M$, each one of these points is uniformly and independently distributed on this interval.

Let us denote now by~$V$ the number of points of that same Poisson process, which belong to~$(u-h,u]$ among the~$N^{+}_{\ome_u}$ points corresponding to the plus-pivotal trajectories for the event~$A$ on configuration~$\ome_u$, and additionally denote by~$V'$ the number of points of this Poisson process which belong to the same interval, $(u-h,u]$, but now among the~$M$ points of the process at level~$u$.

Then the probability that~$A$ ceases to occur under the interlacement at level~$u-h$, given that it occurs at level~$u$, will be equal to the probability that at least one of the~$N^{+}_{\ome_u}$ points of the Poisson process belongs to~$(u-h,u]$, or a collection with at least two among the~$M$ points belong to that interval and the simultaneous removal of all trajectories corresponding to the points in this collection causes~$A^C$ to occur (see Figure~\ref{figure3}).
\begin{figure}[ht]
\centering
\includegraphics[scale=0.7]{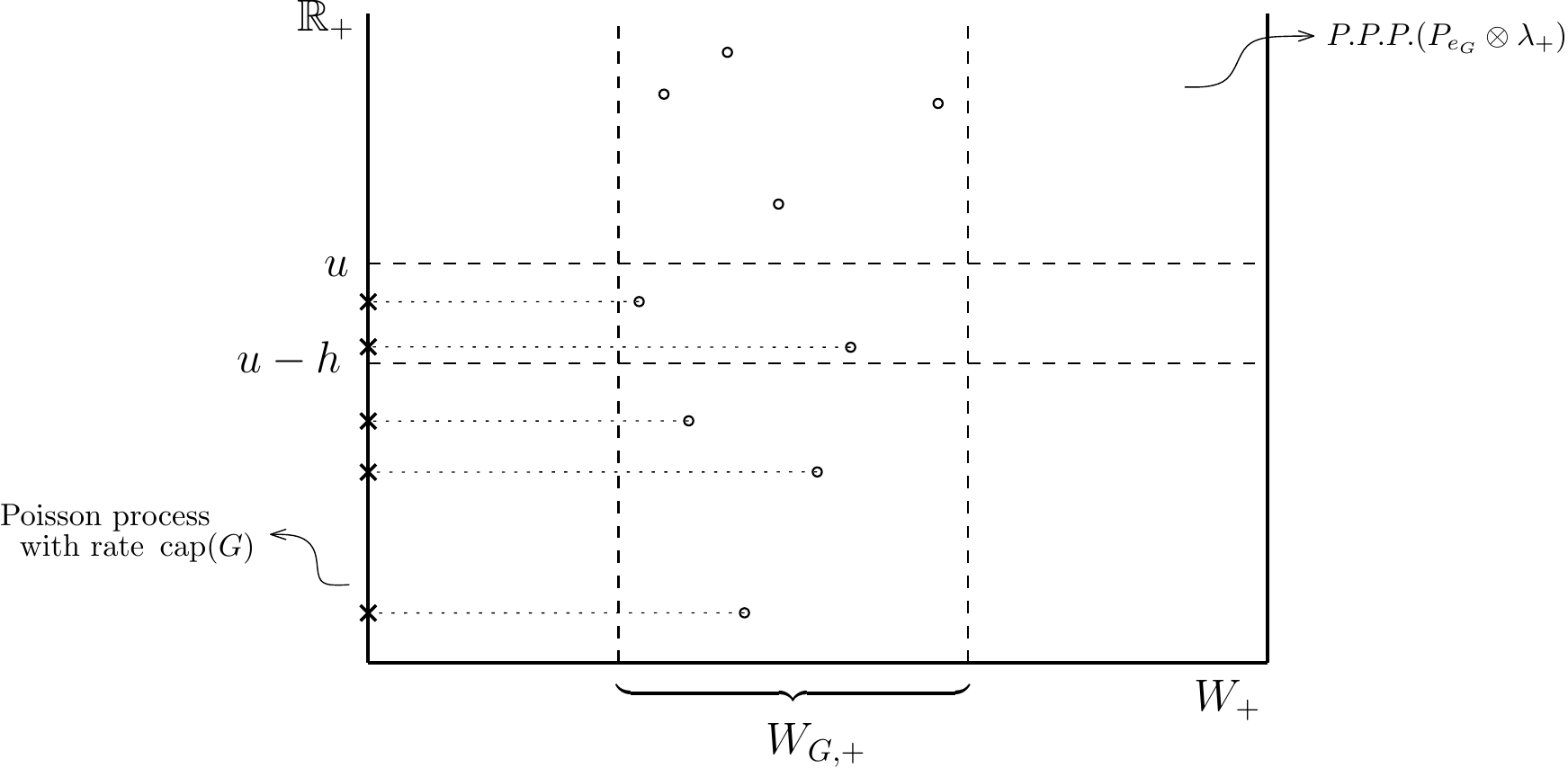}
\caption{\footnotesize By decreasing the level from~$u$ to~$u-h$, some trajectories can be removed from the configuration. If~$A$ occurs at level~$u$ and at least one of the removed trajectories is plus-pivotal, then the event ceases to occur at level~$u-h$. Similarly, if a collection of trajectories with indexes between~$u-h$ and~$u$ causes~$A^C$ to occur when they are simultaneously removed, then the event will also cease to occur at level~$u-h$. Again, recall the spaces~$W_+$ and~$W_{G,+}$, defined in Section~\ref{definitions}. We write~$P.P.P.(P_{e_G}\otimes\lambda_+)$ for the Poisson point process on~$W_+\times\R_+$ with intensity measure~$P_{e_G}\otimes\lambda_+$.}
\label{figure3}
\end{figure}

Precisely, for~$\ome_u=\sum_{i\geq 1}\delta_{(w_i,u_i)}$ in~$\Omega_{u}^{G,+}$, let us consider the event
\begin{align*}
B(V') := \Big\{ \ind_A(\ome_u)=1, \ \  \ind_A\Big(\ome_u-\sum_{i: u-h<u_i\leq u}\delta_{(w_i,u_i)}\Big)=0\Big\},
\end{align*}
that is, the event on which~$A$ occurs under the process at level~$u$, but it ceases to occur if all trajectories with indexes between~$u-h$ and~$u$ are simultaneously removed, noting that the number of these trajectories that are removed is just~$V'$.

Then, according to the definitions of the random variables~$V$ and~$V'$, we have, on~$\{\ome_u\in A\}$,
\begin{align*}
\PP(\ome_{u-h}\notin A \mid M=m, N^{+}_{\ome_u}=n, V=k, \ome_u) =\left\{
\begin{array}{ll}
p^{m,n,\ome_u}, & \mbox{ if } k=0, \\
1, & \mbox{ if } k\geq 1, 
\end{array}
\right.
\end{align*}
where~$p^{m,n,\ome_u} := \PP\Big(\{V'\geq 2\}\cap B(V') \mid M=m, N^{+}_{\ome_u}=n, V=0, \ome_u \Big)$, so that
{\allowdisplaybreaks
\begin{align*}
\PP(\ome_{u-h}\notin A \mid \ome_u) &=  \sum_{m=0}^{\infty}\sum_{n=0}^m\sum_{k=1}^n \PP(M=m, N^{+}_{\ome_u}=n, V=k \mid \ome_u) \\*
& \hspace{1.5cm} + \sum_{m=0}^{\infty}\sum_{n=0}^m  \PP(M=m, N^{+}_{\ome_u}=n, V=0 \mid \ome_u)p^{m,n,\ome_u} \\*
&= \PP\Big(V\geq 1 \mid \ome_u\Big) + \PP\Big(\{V=0\}\cap\{V'\geq 2\}\cap B(V') \mid \ome_u \Big),
\end{align*}}
on~$\{\ome_u\in A\}$.

The first term in the last expression can be calculated by taking into account the above comments about the uniform distribution, in such a way that, on~$\{\ome_u\in A\}$,
\begin{align}
\PP(V\geq 1 \mid \ome_u) =  \sum_{k=1}^{\infty} \binom{N^{+}_{\ome_u}}{k}\Big(\frac{h}{u}\Big)^k \Big(1-\frac{h}{u}\Big)^{N^{+}_{\ome_u}-k} \ind_{\{N^{+}_{\ome_u}\geq k\}}.  
\label{eq41}
\end{align}

Thus we can finally write
\begin{align*}
\PP(\ome_{u-h}\in A \mid \ome_u) &= \ind_{\{\ome_u\in A\}} - \sum_{k=1}^{\infty} \binom{N^{+}_{\ome_u}}{k}\Big(\frac{h}{u}\Big)^k \Big(1-\frac{h}{u}\Big)^{N^{+}_{\ome_u}-k} \ind_{\{N^{+}_{\ome_u}\geq k\}}\ind_{\{\ome_u\in A\}}\\
& \hspace{4cm} -\PP\Big(\{V=0\}\cap\{V'\geq 2\}\cap B(V') \mid \ome_u \Big)\ind_{\{\ome_u\in A\}},
\end{align*}
and taking the expectation with respect to~$\PP^u$, by the Monotone Convergence Theorem we have
\begin{align}
\PP^{u-h}(A) &= \PP^u(A) - \sum_{k=1}^{\infty} \Big(\frac{h}{u}\Big)^k \EE^u\Big[\binom{N^{+}_{\ome_u}}{k} \Big(1-\frac{h}{u}\Big)^{N^{+}_{\ome_u}-k} \ind_{\{N^{+}_{\ome_u}\geq k\}}\ind_{\{\ome_u\in A\}}\Big] \nonumber\\
& \hspace{3cm} -\EE^u\Big[\PP\Big(\{V=0\}\cap\{V'\geq 2\}\cap B(V') \mid \ome_u \Big)\ind_{\{\ome_u\in A\}}\Big]. 
\label{eq42}
\end{align}

At this point, we need the next two lemmas.

\begin{lem} 
\begin{align*} 
r_h := \sum_{k=2}^{\infty} \Big(\frac{h}{u}\Big)^k \EE^u\Big[\binom{N^{+}_{\ome_u}}{k} \Big(1-\frac{h}{u}\Big)^{N^{+}_{\ome_u}-k} \ind_{\{N^{+}_{\ome_u}\geq k\}}\ind_{\{\ome_u\in A\}}\Big] = o(h), \ \ \mbox{ when } h\downarrow 0 .
\end{align*}
\label{lemma41}
\end{lem}

\begin{proof} Since~$h<u$, we have
\begin{align*}
\binom{N^{+}_{\ome_u}}{k} \Big(1-\frac{h}{u}\Big)^{N^{+}_{\ome_u}-k} \ind_{\{N^{+}_{\ome_u}\geq k\}}\ind_{\{\ome_u\in A\}} \leq \binom{N^{+}_{\ome_u}}{\lfloor N^{+}_{\ome_u}/2\rfloor},
\end{align*}
where~$\lfloor b \rfloor$ is the integer part of~$b$, and by monotonicity of expectation~$\EE^u[~\cdot~]$, 
\begin{align*}
r_h \leq \EE^u\Big[\binom{N^{+}_{\ome_u}}{\lfloor N^{+}_{\ome_u}/2\rfloor}\Big]\sum_{k=2}^{\infty} \Big(\frac{h}{u}\Big)^k = \EE^u\Big[\binom{N^{+}_{\ome_u}}{\lfloor N^{+}_{\ome_u}/2\rfloor}\Big] \frac{\Big(\frac{h}{u}\Big)^2}{1-\Big(\frac{h}{u}\Big)}.
\end{align*}

But~$N^{+}_{\ome_u}\leq M$, which implies that
\begin{align*}
\binom{N^{+}_{\ome_u}}{\lfloor N^{+}_{\ome_u}/2\rfloor} \leq \binom{M}{\lfloor M/2\rfloor},
\end{align*}
and the expectation of the term on the right is finite, because~$M$ is Poisson distributed.

This concludes the proof of Lemma~\ref{lemma41}.
\end{proof}

\begin{lem} 
\begin{align*} 
r_h':= \EE^u\Big[\PP\Big(\{V=0\}\cap\{V'\geq 2\}\cap B(V') \mid \ome_u \Big)\ind_{\{\ome_u\in A\}}\Big]=o(h), \ \ \mbox{ when } h\downarrow 0.
\end{align*}
\label{lemma42}
\end{lem}

\begin{proof} First of all note that
\begin{align}
\PP\Big(\{V=0\}\cap\{V'\geq 2\}\cap B(V') \mid \ome_u \Big)\ind_{\{\ome_u\in A\}} \leq \PP(V'\geq 2 \mid \ome_u),
\label{eq43}
\end{align}
and using a similar reasoning to that employed when establishing~\eqref{eq41}, one can see that
\begin{align*}
\PP(V' \geq 2 \mid \ome_u ) =  \sum_{k=2}^{\infty} \binom{M}{k}\Big(\frac{h}{u}\Big)^k \Big(1-\frac{h}{u}\Big)^{M-k}\ind_{\{M\geq k\}}.
\end{align*}

Then, by taking the expectation~$\EE^u$ in expression~\eqref{eq43}, and using the Monotone Convergence Theorem once again, we obtain
\begin{align*}
r_h' \leq \EE^u\Big[P(V' \geq 2 \mid \ome_u )\Big] = \sum_{k=2}^{\infty}\Big(\frac{h}{u}\Big)^k \EE^u\Big[\binom{M}{k} \Big(1-\frac{h}{u}\Big)^{M-k} \ind_{\{M\geq k\}}\Big],
\end{align*}
from where we conclude that
\begin{align*}
r_h' \leq \EE^u\Big[\binom{M}{\lfloor M/2\rfloor}\Big] \frac{\Big(\frac{h}{u}\Big)^2}{1-\Big(\frac{h}{u}\Big)},
\end{align*}
which shows our claim.
\end{proof}

Now observe that
\begin{align*} 
\EE^u\Big[\binom{N^{+}_{\ome_u}}{1} \Big(1-\frac{h}{u}\Big)^{N^{+}_{\ome_u}-1} \ind_{\{N^{+}_{\ome_u}\geq 1\}}\ind_{\{\ome_u\in A\}}\Big] = \EE^u\Big[N^{+}_{\ome_u}\Big(1-\frac{h}{u}\Big)^{N^{+}_{\ome_u}-1} \ind_{\{N^{+}_{\ome_u}\geq 1\}}\Big], 
\end{align*}
and, again by the Monotone Convergence Theorem,
\begin{align*}
\lim_{h\downarrow 0} \EE^u\Big[N^{+}_{\ome_u}\Big(1-\frac{h}{u}\Big)^{N^{+}_{\ome_u}-1} \ind_{\{N^{+}_{\ome_u}\geq 1\}}\Big] = \EE^u[N^{+}_{\ome_u}\ind_{\{N^{+}_{\ome_u}\geq 1\}}] = \EE^u[N^{+}_{\ome_u}],
\end{align*}
where the last equality comes from the fact that~$N^{+}_{\ome_u}$ only takes non-negative integer values.

Finally, using~\eqref{eq42} combined with the two lemmas, we conclude that
\begin{align*}
\lim_{h\downarrow 0}\frac{\PP^u(A)-\PP^{u-h}(A)}{h} = \lim_{h\downarrow 0} \frac{\frac{h}{u}\EE^u\Big[N^{+}_{\ome_u}\Big(1-\frac{h}{u}\Big)^{N^{+}_{\ome_u}-1} \ind_{\{N^{+}_{\ome_u}\geq 1\}}\Big] +r_h+r_h'}{h} = \frac{1}{u}\EE^u[N^{+}_{\ome_u}],
\end{align*}
thus establishing expression~\eqref{RFRI_2}.

\end{subsection}

\begin{subsection}{Another calculation for the right derivative} \label{right_derivative2}

In this section, we discuss the calculation of the right derivative of~$\PP^u(A)$ in a slightly different way from that presented previously in Section~\ref{right_derivative}, and this will lead us to obtain expression~\eqref{RFRI_3}, thus completing the proof of Theorem~\ref{RFRI_Thm}.

Let us focus again on the calculation of~$\PP^{u+h}(A)$, now for some~$u> 0$ and~$0<h<u$.

Recall the random variables~$M$ and~$H$, which are Poisson distributed with parameters respectively equal to~$u\capa(G)$ and~$h\capa(G)$, and which are also independent.

Conditioning on~$M$, we can write
\begin{align}
\PP^u(A) = \sum_{\ell=0}^{\infty} P_\ell(A)\frac{e^{-u\capa(G)}(u\capa(G))^\ell}{\ell!},
\label{eq44}
\end{align} 
where~$P_\ell(A):=\PP^u(A\mid M=\ell)$ does not depend on~$u$, but only on~$\ell$.

Now conditioning on the variables~$M$ and~$H$, and denoting the pro\-ba\-bi\-li\-ty~$\PP(\ome_{u+h}\in A\mid M=m, H=k)$ by~$P_{m+k}(A)$, we obtain
{\allowdisplaybreaks
\begin{align*}
\PP^{u+h}(A) &=  \sum_{m=0}^{\infty}\sum_{k=0}^{\infty}P_{m+k}(A) \PP(M=m)\PP(H=k) \\  
&=  \sum_{\ell=0}^{\infty}\sum_{k=0}^{\ell}P_{\ell}(A) \PP(M=\ell-k)\PP(H=k)  \\
&=  \sum_{\ell=0}^{\infty}\Big\{P_{\ell}(A)\frac{e^{-u\capa(G)}(u\capa(G))^{\ell}}{\ell!}\Big[\sum_{k=0}^{\ell} \binom{\ell}{k} \Big(\frac{h}{u}\Big)^k e^{-h\capa(G)}\Big]\Big\}, 
\end{align*}}
observing that~$P_{\ell}(A)$ also does not depend on~$h$.

Then, by subtracting~$\PP^u(A)$, as represented in~\eqref{eq44}, from the last expression, we have
\begin{align*}
\PP^{u+h}(A)-\PP^u(A) &= \sum_{\ell=0}^{\infty}\Big\{P_{\ell}(A)\frac{e^{-u\capa(G)}(u\capa(G))^{\ell}}{\ell!}\Big[e^{-h\capa(G)} - 1\Big]\Big\}    \\
& \hspace{0.5cm} +~ \sum_{\ell=1}^{\infty}\Big\{P_{\ell}(A)\frac{e^{-u\capa(G)}(u\capa(G))^{\ell}}{\ell!}\Big[e^{-h\capa(G)} \frac{\ell h}{u}\Big]\Big\} \\
& \hspace{0.5cm} +~ \sum_{\ell=2}^{\infty}\Big\{P_{\ell}(A)\frac{e^{-u\capa(G)}(u\capa(G))^{\ell}}{\ell!}\Big[e^{-h\capa(G)}\sum_{k=2}^{\ell} \binom{\ell}{k} \Big(\frac{h}{u}\Big)^k \Big]\Big\}.
\end{align*}

So, in order to compute the right derivative of~$\PP^u(A)$, it suffices to divide this expression by~$h$ and then to take the limit when~$h$ tends to zero. 

But the third summation in the last expression is equal to~$o(h)$ when~$h\downarrow 0$. To see this, observe that, since~$e^{-h\capa(G)}\leq 1$ for all~$h\geq 0$, and the terms~$P_{\ell}(A)$ are such that~$0\leq P_{\ell}(A)\leq 1$ for all~$\ell\geq 0$, $\ell\in\mathbb{N}$, then the referred summation is less than or equal to
\begin{align*}
\frac{\Big(\frac{h}{u}\Big)^2}{1-\Big(\frac{h}{u}\Big)} \sum_{\ell=0}^{\infty}\Big\{\frac{e^{-u\capa(G)}(u\capa(G))^{\ell}}{\lfloor \ell/2\rfloor!(\ell-\lfloor \ell/2\rfloor)!} \Big\},
\end{align*}
and it is possible to show that the last series converges, using for example the Ratio Test.

Thus we can write
\begin{align*}
\lim_{h\downarrow 0}\frac{\PP^{u+h}(A)-\PP^u(A)}{h} = -\capa(G)\PP^u(A) + \frac{1}{u}\sum_{\ell=0}^{\infty}\ell P_{\ell}(A)\frac{e^{-u\capa(G)}(u\capa(G))^{\ell}}{\ell!},
\end{align*}
and the last summation just corresponds to the expectation of~$M\PP^u(A\mid M)$, with respect to the law~$\PP^u$, which, in turn, is equal to~$\EE^u[ M\ind_{\{\ome_u\in A\}} ]$. This leads to expression~\eqref{RFRI_3}, finally completing the proof of Theorem~\ref{RFRI_Thm}.

\end{subsection}

\end{section}

\vspace{1cm}

\textit{Acknowledgments.} Diego F. de Bernardini was partially supported by Coordena\c{c}\~ao de Aperfei\c{c}oamento de Pessoal de N\'ivel Superior (CAPES) and by grant \#2014/14323--9, S\~ao Paulo Research Foundation (FAPESP). Serguei Popov thanks Conselho Nacional de Desenvolvimento Cient\'ifico e Tecnol\'ogico (CNPq) (grant \#300886/2008--0) and S\~ao Paulo Research Foundation (FAPESP) (grant \#2009/52379--8) for partial support. 
The authors thank the referees for the valuable comments and suggestions.

\end{document}